\documentclass[10pt,a4paper]{amsart}
\usepackage[alphabetic]{amsrefs}
\usepackage{enumerate,booktabs,amssymb}
\usepackage{tikz}
\usetikzlibrary{cd}

\DeclareMathOperator{\Bl}{Bl}
\DeclareMathOperator{\Spec}{Spec}
\DeclareMathOperator{\Frac}{Frac}
\DeclareMathOperator{\Pic}{Pic}
\DeclareMathOperator{\Cl}{Cl}
\DeclareMathOperator{\SL}{SL}
\DeclareMathOperator{\Fix}{Fix}
\DeclareMathOperator{\Image}{Im}
\DeclareMathOperator{\Ann}{Ann}
\DeclareMathOperator{\Sing}{Sing}
\DeclareMathOperator{\divisor}{div}
\DeclareMathOperator{\codim}{codim}
\DeclareMathOperator{\charac}{char}
\DeclareMathOperator{\Ker}{Ker}
\DeclareMathOperator{\Gal}{Gal}
\DeclareMathOperator{\Dih}{Dih}
\DeclareMathOperator{\Tet}{Tet}
\DeclareMathOperator{\Oct}{Oct}
\DeclareMathOperator{\Ico}{Ico}
\newcommand{\BinDih}{\widetilde \Dih}
\newcommand{\BinTet}{\widetilde \Tet}
\newcommand{\BinOct}{\widetilde \Oct}
\newcommand{\BinIco}{\widetilde \Ico}
\newcommand{\id}{\mathrm{id}}
\newcommand{\sm}{\mathrm{sm}}
\newcommand{\actson}{\curvearrowright}
\newcommand{\from}{\leftarrow}
\newcommand{\divides}{\mid}
\newcommand{\notdivides}{\nmid} 
\newcommand{\namedto}[1]{\xrightarrow{#1}}
\newcommand{\thpower}[2]{{#1}^{(#2)}}
\newcommand{\pthpower}[1]{\thpower{#1}{p}}
\newcommand{\secondpower}[1]{\thpower{#1}{2}}
\newcommand{\divisorialfix}[1]{(#1)}
\newcommand{\isolatedfix}[1]{\langle #1 \rangle}
\newcommand{\set}[1]{\{#1\}}
\newcommand{\positive}[1]{(#1)^{+}}
\newcommand{\partialdd}[1]{\frac{\partial}{\partial #1}}
\newcommand{\restrictedto}[1]{\rvert_{#1}}
\newcommand{\abs}[1]{\lvert #1 \rvert}
\newcommand{\spanned}[1]{\langle #1 \rangle}
\newcommand{\smpoint}{A_0}
\newcommand{\C}{\mathrm{C}}
\newcommand{\pietloc}{\pi^{\mathrm{et}}_{\mathrm{loc}}}
\newcommand{\piet}{\pi^{\mathrm{et}}}

\newcommand{\idealm}{\mathfrak{m}}
\newcommand{\idealn}{\mathfrak{n}}

\newcommand{\cO}{\mathcal O}
\newcommand{\bZ}{\mathbb Z}
\newcommand{\bF}{\mathbb F}
\newcommand{\bP}{\mathbb P}

\newcommand{\map}[4][\to]{#2 \colon #3 #1 #4}

 \theoremstyle{plain}
 \newtheorem{thm}{Theorem}[section]
 \newtheorem{lem}[thm]{Lemma}
 \newtheorem{prop}[thm]{Proposition}
 \newtheorem{cor}[thm]{Corollary}
 \newtheorem{claim}[thm]{Claim}
 \theoremstyle{definition}
 \newtheorem{rem}[thm]{Remark}
 \newtheorem{defn}[thm]{Definition}
 \newtheorem{example}[thm]{Example}
 \newtheorem{conv}[thm]{Convention}

\title
[Purely inseparable coverings of RDPs in positive characteristic]
{Purely inseparable coverings of rational double points in positive characteristic}

\author{Yuya Matsumoto}
\date{2022/04/03}

\address{Department of Mathematics, Faculty of Science and Technology, Tokyo University of Science, 2641 Yamazaki, Noda, Chiba, 278-8510, Japan}
\email{\url{matsumoto_yuya@ma.noda.tus.ac.jp}}
\email{\url{matsumoto.yuya.m@gmail.com}}


\thanks{This work was supported by JSPS KAKENHI Grant Number 16K17560 and 20K14296.}

\begin{document}

\begin{abstract}
We classify purely inseparable morphisms of degree $p$
between rational double points (RDPs) in characteristic $p > 0$.
Using such morphisms, we refine a result of Artin that any RDP admits a finite smooth covering.
\end{abstract}

\maketitle

\section{Introduction}

Rational double points, RDPs for short, 
are perhaps the most basic singularities in dimension $2$.
They are precisely the canonical singularities in dimension $2$
and have many other characterizations. 
In characteristic $0$, they are characterized as the quotient singularities by the linear actions of finite subgroups of $\SL_2$,
and in particular they admit finite coverings by regular local rings.
However, some of the quotient presentations fail in positive characteristic,
even if we take finite group schemes into account.
Nevertheless, Artin \cite{Artin:RDP} determined \'etale fundamental groups of RDPs and showed the existence of finite coverings (possibly ramified and possibly inseparable) by regular local rings.

In this paper we study
finite purely inseparable morphisms $\Spec B \to \Spec B'$ of degree $p$
between RDPs,
and give a classification (Theorem \ref{thm:main}).
As an application, we show (Theorem \ref{thm:smooth covering}) that any RDP $\Spec B$ admits a finite covering of the form $\Spec k[[x,y]] \namedto{f} \Spec \tilde{B} \namedto{g} \Spec B$,
where $f$ is a composite of purely inseparable morphisms of degree $p$ and $g$ is the universal covering outside the closed point.
This refines the above-mentioned result of Artin in the respect that we use a restricted class of morphisms.
We moreover show that, in all but one cases, $\tilde{B}$ and all the intermediate steps of $f$ can be taken to be RDPs.

Our classification will be used in a subsequent paper \cite{Liedtke--Martin--Matsumoto:RDPtors} 
to show that certain RDPs in characteristic $2$ and $3$ 
are not quotient singularities by any finite group schemes.
See loc.\ cit.\ for the precise definition of the quotient singularities by finite group schemes.
 
\bigskip

To state the main theorem, we first introduce some terminology and conditions.

We work over an algebraically closed field $k$ of characteristic $p > 0$.
An \emph{at most RDP} is a complete local $k$-algebra of dimension $2$
that is either an RDP or smooth.

Let $\map{\pi}{\Spec B}{\Spec B'}$ be a finite purely inseparable morphism of degree $p$ between at most RDPs.
Then it is given as the quotient by a $p$-closed derivation $D$ on $B$ (Proposition \ref{prop:inseparable extension and p-closed derivation}(\ref{item:inseparable then p-closed})), 
and $D$ is unique up to $\Frac(B)^*$-multiple.
The relation between $B$ and $B'$ is reciprocal in the sense that
knowing $\map{\pi}{\Spec B}{\Spec B'}$ is equivalent to knowing $\map{\pi'}{\Spec B'}{}{\Spec \pthpower{B}}$.
The morphism $\pi'$ is also purely inseparable of degree $p$ 
and hence given as the quotient by a $p$-closed derivation $D'$ on $B'$.
We say that the morphisms $\pi$ and $\pi'$ are dual to each other.

We may assume that the fixed locus $\Fix(D)$ of $D$ (Definition \ref{def:fixed locus}), which is a closed subscheme of $\Spec B$,
satisfies exactly one of the following three conditions:
\begin{enumerate}[(a)]
	\item \label{case:empty} $\Fix(D) = \emptyset$.
	\item \label{case:point} $\Fix(D) = \set{\idealm}$.
	\item \label{case:non-principal} The divisorial part $\divisorialfix{D}$ of $\Fix(D)$ is a non-Cartier (Weil) divisor.
	In other words, the corresponding ideal of $B$ is not principal.
\end{enumerate}
Similarly we may assume exactly one of 
(\ref{case:empty}$'$), (\ref{case:point}$'$), (\ref{case:non-principal}$'$) according to $\Fix(D')$.
Thus there are a priori $3^2 = 9$ possibilities.

\begin{thm} \label{thm:main}
Suppose $B$ and $B'$ are complete $2$-dimensional local $k$-algebras,
both either smooth or RDP.
Suppose $\Spec B \to \Spec B'$ is purely inseparable of degree $p$. 
Then, (\ref{case:non-principal}) holds if and only if (\ref{case:non-principal}$'$) holds.
Therefore $5$ possibilities among $9$ remains,
and in each case we have the following classification.
\begin{enumerate}
	\item \label{case:smooth}
	If (\ref{case:empty}) and (\ref{case:empty}$'$) hold,
	then both $B$ and $B'$ are smooth.
	In this case, it is known 
	(Rudakov--Shafarevich \cite{Rudakov--Shafarevich:inseparable}*{Theorem 1 and Corollary}
	and Ganong \cite{Ganong:frobenius}*{Theorem})
	that $B = k[[x,y]]$ and $B' = k[[x,y^p]]$ for some $x,y \in B$.
	
	\item \label{case:non-fixed}
	If (\ref{case:empty}) and (\ref{case:point}$'$) hold,
	then $B \cong k[[x,y,z]]/(F) \supset B' \cong k[[x,y,z^p]]/(F)$ and $D(x) = D(y) = 0$,
	where $F$ is one of Table \ref{table:derivations on RDPs:non-fixed}.
	
	\item \label{case:non-fixed2}
	Dually, if (\ref{case:point}) and (\ref{case:empty}$'$) hold,
	then either $B \cong k[[x,y]]$ or $B \cong k[[x,y,z]]/(F)$
	with $F$ and $D$ as in one of Table \ref{table:derivations on RDPs:non-fixed2},
	up to replacing $D$ with a unit multiple.
	$B'$ is generated 
	by $x^p,y^p$ and one more element if $B$ is smooth
	and by $x^p,y^p,z$ if $B$ is an RDP,
	subject to an equation similar to the one in the dual case in Table \ref{table:derivations on RDPs:non-fixed}.
	
	\item \label{case:both fixed:principal}
	Suppose (\ref{case:point}) and (\ref{case:point}$'$) hold.
	Then $(\Sing(B), \Sing(B'))$ is one of Table \ref{table:derivations on RDPs:fixed},
	and there are elements $x,y,z,w \in B$ such that 
	\begin{align*}
	B  &= k[[x,y,z,w]]     / (x^p - P(y^p,z,w), w - Q(z,y,x)) \quad \text{and} \\
	B' &= k[[w,z,y^p,x^p]] / (w^p - Q(z,y,x)^p, x^p - P(y^p,z,w)),
	\end{align*}
	with $P$ and $Q$ as in the table up to terms of high degree.
	Up to a unit multiple, $D$ satisfies $(D(x),D(y),D(z)) = (-Q_y, Q_x, 0)$.
	
	\item \label{case:both fixed:non-principal}
	Suppose (\ref{case:non-principal}) and (\ref{case:non-principal}$'$) hold.
	Let $l^{(\prime)}$ be the order of the class $[\divisorialfix{D^{(\prime)}}]$ in $\Cl(B^{(\prime)})$.
	Then $l = l'$ and $l \divides (p-1)$, 
	and we have $B \cong k[[x,y,z]]/(F)$ with $D(z) = 0$ and $B' \cong k[[w,z,y^p]]/(F')$, where
	$(l, \Sing(B), \Sing(B'), F, w)$ is one of Table \ref{table:derivations on RDPs:non-principal}. 
	
\end{enumerate}
In particular, if $B'$ (resp.\ $B$) is smooth, then (\ref{case:empty}) (resp.\ (\ref{case:empty}$'$)) holds.
\end{thm}

In Table \ref{table:derivations on RDPs:fixed},
$\positive{q} := \max\set{0, q}$ denotes the positive part of a real $q$.
In Table \ref{table:derivations on RDPs:non-principal},
$l$ is an arbitrary positive integer and $p$ is a prime satisfying $p \equiv 1 \pmod{l}$, unless specified otherwise.

\begin{conv} \label{conv:Dnr}
	Non-taut RDPs are the RDPs (over a fixed algebraically closed field $k$) not uniquely determined from the dual graph of the minimal resolution.
	Artin \cite{Artin:RDP} classified them and introduced the notation $D_n^r$, $E_n^r$ using the coindex $r$.
	We follow Artin's notation 
	with the following exception: 
	We say that $k[[x,y,z]] / (z^2 + x^2 y + z y^n + z x y^{n-s})$
	($n \geq 2$, $0 \leq s \leq n-1$) in characteristic $2$ to be of type $D_{2n+1}^{s + 1/2}$,
	instead of Artin's notation $D_{2n+1}^s$.
	Consequently, the range of $r$ for $D_{2n+1}^r$ is $\set{\frac{1}{2}, \frac{3}{2}, \dots, \frac{2n-1}{2}}$.
	Under this notation, the RDP defined by $z^2 + x^2 y + z x y^m + y^n = 0$ ($m \geq 1$, $n \geq 2$)
	is of type $D_{n+2m}^{n/2}$ regardless of the parity of $n$,
	and this makes it easier to write Table \ref{table:derivations on RDPs:fixed} and 
	the proofs of Theorem \ref{thm:smooth covering} and Proposition \ref{prop:smooth covering etale last}. 
	
	Also, we use the convention that $D_3 := A_3$ and that $\smpoint$ is a smooth point.
\end{conv}

\begin{conv} \label{conv:RDP}
	Let $X = \Spec B$ be the spectrum of a complete local $k$-algebra $B$,
	$x$ its closed point, $U = X \setminus \set{x}$, 
	and suppose $U$ is smooth.
\begin{enumerate}
\item
	$\Cl(B) := \Pic(U)$ is the (divisor) class group of $B$.
	If $B$ is an at most RDP (which is complete by our definition),
	it follows from \cite{Lipman:rationalsingularities}*{Proposition 17.1} that
	this group is determined from the Dynkin diagram as in Table \ref{table:Picard group of RDP} and is independent of the characteristic and the coindex. 
\item \label{conv:RDP:universal covering}
	We call $\pietloc(B) := \piet_1(U)$ the local fundamental group of $B$,
	and we define the universal (\'etale) covering of RDP $X = \Spec B$ to be the integral closure $X'$ of $X$
	in the universal covering $U'$ of $U$.
	Artin \cite{Artin:RDP} showed that if $X$ is an RDP then $\pietloc(B)$ is always finite and $X'$ is an at most RDP, 
	and determined them (Table \ref{table:pi1}).
	We say that $X$ is simply-connected if $U$ is so.
\end{enumerate}
\end{conv}

We prove Theorem \ref{thm:main} in Section \ref{sec:proof}
after some preparations on derivations in Section \ref{sec:derivation}.

\begin{table} 
	\caption{$p$-closed derivations on RDPs 
		with $D$ fixed-point-free ($m \geq 2$)} \label{table:derivations on RDPs:non-fixed} 
		\begin{tabular}{llll}
			\toprule
			$p$ & $B$            & $B'$      & equation of $B$            \\ 
			\midrule                      
			any & $A_{p-1}$      & smooth    & $x y + z^p$                \\ 
			any & $A_{pm-1}$     & $A_{m-1}$ & $x y + z^{pm}$             \\
			\midrule                      
			$5$ & $E_8^0$        & smooth    & $z^5 + x^2 + y^3$          \\ 
			\midrule                      
			$3$ & $E_6^0$        & smooth    & $z^3 + x^2 + y^4$          \\ 
			$3$ & $E_7^0$        & $A_1$     & $x^2 + y^3 + y z^3$        \\
			$3$ & $E_8^0$        & smooth    & $z^3 + x^2 + y^5$          \\ 
			\midrule                      
			$2$ & $D_{2m}^0$     & smooth    & $z^2 + x^2 y + x y^{m}$    \\ 
			$2$ & $D_{2m+1}^{1/2}$ & $A_{1}$ & $x^2 + y z^2 + x y^{m}$     \\
			$2$ & $E_{6}^{0}$    & $A_{2}$   & $x^2 + x z^2 + y^3$        \\
			$2$ & $E_7^0$        & smooth    & $z^2 + x^3 + x y^3$        \\ 
			$2$ & $E_8^0$        & smooth    & $z^2 + x^3 + y^5$          \\ 
			\bottomrule
		\end{tabular}
\end{table}

\begin{table} 
	\caption{$p$-closed derivations on smooth points and RDPs whose quotients are RDPs,
		with $D'$ fixed-point-free ($m \geq 2$)} \label{table:derivations on RDPs:non-fixed2} 
	\begin{tabular}{llllll}
		\toprule                                                                                                  
			$p$ & $B$            & $B'$            & equation of $B$      & $D(x), D(y), D(z)$       & $h$         \\ 
			\midrule                                                               
			any & smooth         & $A_{p-1}$       & ---                  & $x, -y$                  & $1$         \\ 
			any & $A_{m-1}$      & $A_{pm-1}$      & $x y + z^{m}$        & $x, -y, 0$               & $1$         \\ 
			\midrule                                                               
			$5$ & smooth         & $E_8^0$         & ---                  & $y^2, x$                 & $0$         \\ 
			\midrule                                                               
			$3$ & smooth         & $E_6^0$         & ---                  & $y^3, x$                 & $0$         \\ 
			$3$ & $A_1$          & $E_7^0$         & $x^2 + y^3 + yz$     & $z, x, 0$                & $0$         \\
			$3$ & smooth         & $E_8^0$         & ---                  & $y^4, -x$                & $- y^3$     \\ 
			\midrule                                                               
			$2$ & smooth         & $D_{2m}^0$      & ---                  & $x^2 + mxy^{m-1}, y^{m}$ & $m y^{m-1}$ \\ 
			$2$ & $A_{1}$        & $D_{2m+1}^{1/2}$& $x^2 + yz + xy^{m}$  & $z+my^{m-1}x, y^{m}, 0$  & $m y^{m-1}$ \\ 
			$2$ & $A_{2}$        & $E_6^0$         & $x^2 + xz + y^3$     & $y^{2}, z, 0$            & $0$         \\ 
			$2$ & smooth         & $E_7^0$         & ---                  & $xy^2, x^2 + y^3$        & $y^2$       \\ 
			$2$ & smooth         & $E_8^0$         & ---                  & $y^4, x^2$               & $0$         \\ 
			\bottomrule
	\end{tabular}
\end{table}

\newcommand{\somethinglonga}{{\scriptsize $z+my^{m-1}x+ny^{n-1},y^{m}$}}
\begin{table} 
	\caption{$p$-closed derivations on RDPs whose quotients are RDPs,
		with $\Fix(D) = \set{\idealm}$ and $\Fix(D') = \set{\idealm'}$ ($m,m' \geq 1$, $n \geq 2$)} 
	\label{table:derivations on RDPs:fixed} 
	{\footnotesize
					
		\begin{tabular}{lllllll}
			\toprule
			$p$ & $B$ & $B'$ & $x^p = P(y^p,z,w)\approx $ & $w = Q(z,y,x) \approx $ & $D(x),D(y) \approx$ & $h \approx$ \\
			\midrule                                                                 
			$3$ & $E_6^1$ & $E_6^1$ & $z^2 + z   w$ & $y^2 + yx$ & $x-y, -y$ & $1$ \\
			$3$ & $E_8^1$ & $E_6^0$ & $z^2 + y^3 w$ & $y^2 + yx$ & $x-y, -y$ & $1$ \\
			$3$ & $E_6^0$ & $E_8^1$ & $z^2 + z   w$ & $y^2 + zx$ & $y, z$    & $0$ \\
			$3$ & $E_8^0$ & $E_8^0$ & $z^2 + y^3 w$ & $y^2 + zx$ & $y, z$    & $0$ \\
			\midrule
			$2$ & $D_{2m'+2}^{\positive{2-m}}$ & $D_{2m+2}^{\positive{2-m'}}$ & $y^2 z + z^{m'} w $  & $zy + y^{m}x$ & $z+my^{m-1}x, y^{m}$        & $m y^{m-1}$ \\ 
			$2$ & $D_{2n}^{\positive{n-m}}$    & $D_{n+2m}^{n/2}$             & $z w$          & $zy + y^n + y^{m}x$ & \somethinglonga                & $m y^{m-1}$ \\ 
			$2$ & $D_{n+2m}^{n/2}$             & $D_{2n}^{\positive{n-m}}$    & $y^2 z + z^n + z^m w$ & $y x$        & $x, y$                      & $1$         \\
			$2$ & $D_{2m+3}^{1/2}$             & $E_7^{\positive{3-m}}$       & $y^2 z    + z^m w$    & $y^3  + z x$ & $y^{2}, z$                  & $0$         \\ 
			$2$ & $E_7^{\positive{3-m}}$       & $D_{2m+3}^{1/2}$             & $z^3      + y^2 w$    & $zy + y^m x$ & $z+my^{m-1}x, y^{m}$        & $m y^{m-1}$ \\ 
			$2$ & $E_7^3$                      & $E_7^3$                      & $z^3      + z   w$    & $y^3  + y x$ & $x+y^2, y$                  & $1$         \\
			$2$ & $E_7^2$                      & $E_8^3$                      & $z^3      + z   w$    & $y^3  + z x$ & $y^{2}, z$                  & $0$         \\ 
			$2$ & $E_8^3$                      & $E_7^2$                      & $z^3      + y^2 w$    & $y^3  + y x$ & $x+y^2, y$                  & $1$         \\
			$2$ & $E_8^2$                      & $E_8^2$                      & $z^3      + y^2 w$    & $y^3  + z x$ & $y^{2}, z$                  & $0$         \\ 
			\bottomrule                                                      
		\end{tabular}
}
\end{table}

\begin{table} 
\caption{$p$-closed derivations on RDPs whose quotients are RDPs, with non-Cartier divisorial fixed loci ($m \geq 1$)
	($p$ is any prime $\equiv 1 \pmod{l}$ unless specified)} \label{table:derivations on RDPs:non-principal} 
{ \footnotesize 
\begin{tabular}{lllllll}
\toprule
$p$ & $B$        & $B'$       & $l$ & $(D(x),D(y))$                  & equation of $B$                & $w$ \\ 
\midrule
    & $A_{l-1}$  & $A_{l-1}$  & any & $(x, ly)$, $(z, lx^{l-1})$     & $- x^{l} + y z$                & $x y^{(p-1)/l}$ \\
    & $D_{m+2}$  & $D_{mp+2}$ & $2$ & $(yz, x)$, $(yx, y^2+z^m)$     & $- x^2 + z (y^2 + z^{m})$      & $x (y^2 + z^m)^{(p-1)/2}$ \\
    & $D_{mp+2}$ & $D_{m+2}$  & $2$ & $(z^2+y^{mp}, 2x)$, $(x, 2y)$  & $- x^2 + y (z^2 + y^{mp})$     & $x y^{(p-1)/2}$ \\
$3$ & $E_7^1$    & $E_7^1$    & $2$ & $(xy,y^2+z)$, $((y^3+z^2)y,x)$ & $- x^2 + (y^3 + z^2)(y^2 + z)$ & $x (y^2 + z)$ \\
\bottomrule                                                                 
\end{tabular}
}
\end{table}

\begin{table}
	\caption{Class groups of complete RDPs (in any characteristic)} \label{table:Picard group of RDP}
	\begin{tabular}{ccccccc}
		\toprule
		smooth & $A_n$ & $D_{2m}$ & $D_{2m+1}$ & $E_6$ & $E_7$ & $E_8$ \\
		\midrule
		$0$ & $\bZ/(n+1)\bZ$ & $(\bZ/2\bZ)^2$ & $\bZ/4\bZ$ & $\bZ/3\bZ$ & $\bZ/2\bZ$ & $0$ \\
		\bottomrule
	\end{tabular}
\end{table}

\section{Preliminaries on $p$-closed derivations} \label{sec:derivation}

\subsection{$p$-closed derivations and their properties}

	A derivation $D$ on an $\bF_p$-algebra $B$ is called \emph{$p$-closed} if there exists $h \in B$ with $D^p = h D$.
(See also Proposition \ref{prop:h in B}.)

The following formula is well-known.
See \cite{Matsumura:commutativeringtheory}*{Theorem 25.5} for a proof.
\begin{lem}[Hochschild's formula]
	Let $B$ be an $\bF_p$-algebra, $a \in B$ an element, and $D$ a derivation on $B$.
	Then 
	\[
	(aD)^p = a^p D^p + (aD)^{p-1}(a) D.
	\]
\end{lem}
\begin{cor} \label{cor:h unit}
	Suppose a derivation $D$ on $B$ is $p$-closed.
\begin{itemize}
\item For any $a \in B$, $aD$ is also $p$-closed.
\item If $B$ is a domain, 
$D \neq 0$, $a \neq 0$, and $(aD)^p = h_1 aD$,
then $h_1 \in (h) + (\Image D)$.
\end{itemize}
\end{cor}
\begin{proof}
Suppose $D^p = h D$ with $h \in B$.
Then it follows from Hochschild's formula that $(aD)^p = h_1 aD$ with $h_1 = a^{p-1} h + D((aD)^{p-2}(a))$.
\end{proof}

We list some properties of $p$-closed derivations and inseparable extensions.
See Section \ref{subsec:examples} for counterexamples under weaker assumptions.
Note that RDPs are normal complete $F$-finite Noetherian domains and hence satisfy all the assumptions.
Here, we say that a domain is normal if it is integrally closed in its field of fractions.

\begin{lem} \label{lem:Dh=0}
	Let $B$ be an $\bF_p$-algebra, $h$ an element of $B$, and $D$ a $p$-closed derivation on $B$ with $D^p = h D$.
	Assume either $B$ is reduced,
	or $\Image D$ contains a regular element of $B$.
	Then $D(h) = 0$.
\end{lem}
\begin{proof}
	We have $h D D = D^p D = D D^p = D(h D) = D(h) D + h D D$,
	hence $D(h) D = 0$.
	If $B$ is reduced, then $D(h)^2 = 0$ and $D(h) = 0$.
	If $D(a) \in \Image D$ is a regular element,
	then $D(h) D(a) = 0$ and $D(h) = 0$.
\end{proof}
	
\begin{prop} \label{prop:inseparable extension and p-closed derivation}
\begin{enumerate}
\item \label{item:p-closed then inseparable}
	Let $B$ be an $F$-finite normal domain and $D \neq 0$ be a $p$-closed derivation on $B$.
	Then $B \supset A := B^D$ is a finite extension of normal domains and $\Frac B / \Frac A$ is purely inseparable of degree $p$.
\item \label{item:inseparable then p-closed}
Let $B \supset A$ be a finite extension of normal domains over $\bF_p$,
such that $\Frac B / \Frac A$ is purely inseparable of degree $p$.
Then there exists a $p$-closed derivation $D$ on $B$ such that $A = B^D$.
Moreover, $D$ is unique up to $(\Frac B)^*$.
\end{enumerate}
\end{prop}
\begin{proof}
(\ref{item:p-closed then inseparable})
Since $B$ is finite over $\pthpower{B}$ (this is the definition of $F$-finiteness),
and clearly $A \supset \pthpower{B}$, it follows that $B$ is finite over $A$.
Normality of $A$ follows from that of $B$.
Now take $b \in \Frac B$ such that $D(b) \neq 0$.
It is straightforward to show that $b^i$ ($0 \leq i < p$) is a basis of $\Frac B$ over $\Frac A$, and that $b^p \in \Frac A$.

(\ref{item:inseparable then p-closed})
There exists $t \in \Frac B$ such that $\Frac B = (\Frac A)[t]$ (and then $t^p \in A$).
Then $D' = \partialdd{t}$ is a $p$-closed derivation on $\Frac B$ such that $\Frac A = (\Frac B)^{D'}$.
Take $x_1, \dots, x_n \in B$ such that $B = A[x_1, \dots, x_n]$.
Take $b \in B \setminus \set{0}$ such that $b D'(x_1^{i_1} \dots x_n^{i_n}) \in B$ for every $i_1, \dots, i_n \in \set{0, 1, \dots, p-1}$.
Then $D := b D'$ takes values in $B$ and is also $p$-closed. 
We obtain $B^D = B \cap \Frac A = A$,
where the last equality follows from $B$ being integral over $A$ and $A$ being normal.

Suppose $D''$ is another derivation on $B$ with the same property.
Then the derivations $(D(t)/D''(t))D''$ and $D$ on $\Frac B$ agree on $\Frac A$ and $t$, hence agree on $\Frac B$.
\end{proof}

\begin{prop} \label{prop:h in B}
Let $B$ be a normal Noetherian domain,
$D \neq 0$ a derivation on $B$, and $h \in \Frac B$.
Suppose $D^p = h D$. Then $h \in B$.

Thus, under this assumption on $B$, the definition of $p$-closedness can be either $h \in B$ or $h \in \Frac B$.
\end{prop}
\begin{proof}
Let $I := (\Image D) \neq 0$ be the ideal generated by the image of $D$.
Then $h I = (\Image h D) = (\Image D^p) \subset (\Image D) = I$.
Since $B$ is Noetherian, $I$ is finitely generated, and then the usual argument using determinant shows that $h$ is integral over $B$.
Since $B$ is normal, we have $h \in B$.
\end{proof}

\subsection{Examples of derivations} \label{subsec:examples}

For the sake of completeness,
we give counterexamples for Lemma \ref{lem:Dh=0} and Propositions \ref{prop:inseparable extension and p-closed derivation} and \ref{prop:h in B} under weaker assumptions.
These examples are not needed for the rest of the paper.
$k$ will be a field of characteristic $p$. 

\begin{example}[cf.\ Lemma \ref{lem:Dh=0}] \label{ex:Dh nonzero}
\begin{enumerate}
\item
	There are examples of $D$ such that $D^p = h_1 D = h_2 D$, $D(h_1) = 0$, but $D(h_2) \neq 0$.
	Let $B = k[t]/(t^p)$, $D = t \partialdd{t}$, $h_1 = 1$, $h_2 = 1 + t^{p-1}$. Then $D(h_1) = 0$ but $D(h_2) = (p-1)t^{p-1} \neq 0$.
\item
	There are examples of $D$ that is $p$-closed but $D(h) \neq 0$ for any $h$ satisfying $D^p = h D$.

	First, let $B = k[x,y] / (x^i y^j \mid i + j = p+1, (i,j) \neq (1,p))$,
	$D = (x^2 + x y) \partialdd{x} + x y \partialdd{y}$.
	We observe that $D$ induces $\map{D_l}{\idealm^l/\idealm^{l+1}}{\idealm^{l+1}/\idealm^{l+2}}$
	and that $D_l$ are injective for $1 \leq l < p$.
	Given $h \in B$, to check whether the derivations $D^p$ and $h D$ are equal, 
	it suffices to check that they agree on $x$ and $y$.
	Since $D^p(x) = D^p(y) = x y^p$, we obtain $D^p = y^{p-1} D$, 
	hence $h$ satisfies this property if and only if $h - y^{p-1} \in \Ann(D(x),D(y)) = (\idealm^p + x \idealm^{p-2})$,
	which implies $h \not\equiv 0 \pmod{\idealm^p}$.
	By the injectivity of $D_{p-1}$, we have $D(h) \neq 0$ for any such $h$.

	Another example is as follows.
	Let $B = k[x, y_2, \dots, y_p] / (y_2, \dots, y_p)^2$,
	$D = y_2 \partialdd{x} + \sum_{i = 2}^{p-1} y_{i+1} \partialdd{y_i} + x y_2 \partialdd{y_p}$.
	We have $D^p = x D$.
	Then $D^p = h D$ if and only if $h \equiv x \pmod{(y_2, \dots, y_p)}$ and then $D(h) \equiv y_2 \not\equiv 0 \pmod{(x y_2, y_3, \dots, y_p)}$.
\end{enumerate}
\end{example}

\begin{example}[cf.\ Proposition \ref{prop:inseparable extension and p-closed derivation}] \label{ex:inseparable but not derivation quotient}
Let $B = k[x_1, \dots, x_n, \dots]$ (which is not $F$-finite).
\begin{enumerate}
	\item 
	Let $\map{D = \sum_{i = 1}^{\infty} x_i \partialdd{x_i}}{B}{B}$.
	Then $\Spec B \to \Spec B^D$ is not finite.
	Indeed, if $\idealm \subset B$ and $\idealm' \subset B^D$ are the maximal ideals at the origin,
	then $\idealm/(\idealm^2 + \idealm')$ is of infinite dimension.
\item
Let $\map{D' = \sum_{i=1}^{\infty} x_i^{-p} \partialdd{x_i}}{B}{\Frac B}$.
Then, although $\Frac B / \Frac B^{D'}$ is purely inseparable of degree $p$ and $B \supset B^{D'}$ is an integral extension between normal domains,
there is no $p$-closed derivation $D$ on $B$ such that $B^D = B^{D'}$.
Indeed, if such $D$ exists then $D = b D'$ for some $b \in \Frac B \setminus \set{0}$,
and $b$ satisfies $b D'(x_i) = b x_i^{-p} \in B$ for every $i$, but this is impossible.
\end{enumerate}
\end{example}

\begin{example}[cf.\ Proposition \ref{prop:h in B}] \label{ex:h not in B}
Let $B = k[t^i \mid i \in \bZ, i > p(p-1)] \subset k[t]$ and $D = t^{p+1} \partialdd{t}$.
This is Noetherian and not normal.
Then $D^p = h D$ with $h = t^{p(p-1)} \in \Frac B$, but $h \notin B$.

Let $B = k[y x^i \mid i \in \bZ]$ or $B = k[x, y x^i \mid i \in \bZ] \subset k[x^{\pm 1}, y]$ and $D = \frac{y}{x} \partialdd{y}$.
This is normal and non-Noetherian.
Then $D^p = h D$ with $h = x^{-(p-1)} \in \Frac B$, but $h \notin B$.
\end{example}

\subsection{$p$-closed derivation quotients}

We recall the notion of the fixed loci of derivations and how they are related to derivation quotients.
Rudakov--Shafarevich \cite{Rudakov--Shafarevich:inseparable}
uses the term \emph{singularity} but 
we do not use this, as we want to distinguish them from the singularities of the varieties.

\begin{defn} \label{def:fixed locus}
	Suppose $D$ is a derivation on a scheme $X$.
	The \emph{fixed locus} $\Fix(D)$
	is the closed subscheme of $X$ 
	corresponding to the sheaf $(\Image (D))$ of ideals generated by $\Image(D) = \set{D(a) \mid a \in \cO_X}$.
	A \emph{fixed point} of $D$ is a point of $\Fix(D)$.
	
	Assume $X$ is a smooth irreducible variety and $D \neq 0$. 
	Then $\Fix(D)$ consists of its divisorial part $\divisorialfix{D}$ and non-divisorial part $\isolatedfix{D}$.
	If we write $D = f \sum_i g_i \partialdd{x_i}$ for some local coordinate $x_i$
	with $g_i$ having no common factor,
	then $\divisorialfix{D}$ and $\isolatedfix{D}$ correspond to the ideals $(f)$ and $(g_i)$ respectively.

	Assume $X$ is a smooth irreducible variety and suppose $D \neq 0$ is now a \emph{rational} derivation,
	locally of the form $f^{-1} D'$ for some regular function $f$ and (regular) derivation $D'$.
	Then we define $\divisorialfix{D} = \divisorialfix{D'} - \divisor(f)$ and $\isolatedfix{D} = \isolatedfix{D'}$.
	
	If $X$ is only normal, then we can still define $\divisorialfix{D}$ as a Weil divisor.
\end{defn}

\begin{prop}[\cite{Matsumoto:k3alphap}*{Proposition 2.14}] \label{prop:n-step}
	Suppose $X_0 \namedto{\pi_0} X_1 \namedto{\pi_1} \dots \namedto{\pi_{n-1}} X_n = \pthpower{X_0}$
	is a sequence of purely inseparable morphisms of degree $p$ between $n$-dimensional integral normal varieties,
	with each $\pi_i$ given by a $p$-closed rational derivation $D_i$ on $X_i$.
	Then $K_{X_0}$ is linearly equivalent to $-\sum_{i = 0}^{n-1} (\pi_{i-1} \circ \dots \circ \pi_0)^*( \divisorialfix{D_i})$.
\end{prop}

\begin{cor} \label{cor:dual derivation}
	If $\map{\pi}{X = \Spec B}{X' = \Spec B'}$ and $D$, $D'$ are as in Theorem \ref{thm:main},
	then the order of the class $[\divisorialfix{D}]$ in $\Cl(B)$
	is equal to that of $[\divisorialfix{D'}]$ in $\Cl(B')$.
\end{cor}
\begin{proof}
	By applying Proposition \ref{prop:n-step} to the sequence $X \to X' \to \pthpower{X}$
	and using $K_X = 0$,
	we obtain $[\divisorialfix{D}] + \pi^*([\divisorialfix{D'}]) = 0$ in $\Cl(B)$.
	Dually we have $\pi'^*([\divisorialfix{\pthpower{D}}]) + [\divisorialfix{D'}] = 0$ in $\Cl(B')$.
\end{proof}

\begin{prop}[\cite{Matsumoto:k3alphap}*{Proposition 2.12}] \label{prop:2-forms}
	Suppose $\map{\pi}{X}{X'}$ is a purely inseparable morphism of degree $p$ between smooth varieties of dimension $m$,
	induced by a $p$-closed rational derivation $D$ such that $\Delta := \Fix(D)$ is divisorial.
	Then we have isomorphisms
	$\Omega^m_{X/k}(\Delta) \cong (\pi^* \Omega^m_{X'/k})(p\Delta)$ and 
	$(\pi_* \Omega^m_{X/k}(\Delta))^D \cong \Omega^m_{X'/k}(\pi_* \Delta)$ 
	sending $f_0 \cdot df_1 \wedge \dots \wedge df_{m-1} \wedge D(g)^{-1} dg 
	\mapsto f_0 \cdot df_1 \wedge \dots \wedge df_{m-1} \wedge D(g)^{-p} d(g^p)$
	if $f_i,g \in \cO_X$, $D(f_i) = 0$ ($1 \leq i \leq m-1$), 
	and $D(g)^{-1} \in \cO_X(\Delta)$, and for the second morphism if moreover $D(f_0) = 0$.

	In particular, we obtain the Rudakov--Shafarevich formula $K_X \sim \pi^* K_{X'} + (p-1) \Delta$.
\end{prop}
See \cite{Matsumoto:k3alphap}*{Section 2.1} for the definition of the action of $D$ on $m$-forms.

\begin{rem}
This isomorphism (although it depends on the choice of $D$) 
may be considered as an analogue of the natural pullback isomorphism 
$\Omega^m_{X/k} \cong \pi^* \Omega^m_{X'/k}(R)$
for a finite flat morphism $\map{\pi}{X}{X'}$ between smooth varieties,
where $R$ is the ramification divisor which is $0$ if and only if $\pi$ is \'etale.
Thus, the existence of $D$ with $\codim \Fix(D) \geq 2$ might be considered as an inseparable analogue of codimension $1$ \'etaleness.
However, this is not a local property, and not compatible with composition, as the following example shows.
Also, if the divisor $\divisorialfix{D}$ is nonzero, then it is not defined canonically from $\pi$ but only up to linear equivalence, 
\end{rem}

\begin{example} \label{ex:kummer}
	Suppose $\charac k = p > 2$.
	Let $\map{\pi}{A_1}{A_2}$ be a purely inseparable isogeny of degree $p$ between abelian surfaces,
	$\map{f_i}{A_i}{X_i = A_i / \set{\pm 1}}$ the quotient morphisms,
	and $\map{\pi'}{X_1}{X_2}$ the purely inseparable morphism of degree $p$ induced by $\pi$.
	Clearly $f_1$ and $f_2$ are \'etale in codimension $1$
	and $\pi$ is induced by a derivation $D$ on $A_1$, which is unique up to scalar and fixed-point-free.
	This $D$ satisfies $[-1]^* D = - D \neq D$ and hence does not descend to $X_1$.
	Hence there is no regular derivation on $X_1$, even on $X_1^{\sm}$, inducing $\pi'$.
\end{example}

\section{Proof of classification} \label{sec:proof}

In this section we prove Theorem \ref{thm:main}.
In Section \ref{subsec:reduction} we reduce the proof to the following 
involved cases, which will be treated in individual sections.

\begin{itemize}
	\item $p = 3$, (\ref{case:point}) and (\ref{case:point}$'$) hold, both $B$ and $B'$ are either of type $E_8^0$ or $E_8^1$, and
	both $h$ and $h'$ are non-units at every upper fixed point (see Section \ref{subsec:reduction} for the definition) of $D$ and $D'$.
	(Section \ref{subsec:3E8})
	\item $p = 2$, (\ref{case:point}) and (\ref{case:point}$'$) hold, and 
	both $h$ and $h'$ are non-units of $B$ and $B'$. 
	(Section \ref{subsec:2})
	\item The divisorial part $\divisorialfix{D}$ of $\Fix(D)$ is not Cartier. 
	(Section \ref{subsec:non-principal})
\end{itemize}

\subsection{Reduction part} \label{subsec:reduction}

We first note that the conditions (\ref{case:empty}), (\ref{case:point}), (\ref{case:non-principal}) are 
pairwise exclusive, and that we can replace $D$ so that one of these holds.
Indeed, replacing $D$ changes $\Fix(D)$ precisely by a Cartier divisor.
Similar for $D'$.

The equivalence (\ref{case:non-principal}) $\iff$ (\ref{case:non-principal}$'$) 
and the equality $l = l'$
follow from Corollary \ref{cor:dual derivation}.
This case will be considered in Section \ref{subsec:non-principal}.
Assume this is not the case.

The following assertions hold.
\begin{enumerate}
\item \label{item:B' smooth then empty} 
If $B'$ is smooth then (\ref{case:empty}) holds
(by \cite{Matsumoto:k3alphap}*{Theorem 3.3(2)}).
\item \label{item:B smooth then empty'}
Dually, if $B$ is smooth then (\ref{case:empty}$'$) holds.
\item \label{item:empty and B RDP} 
Suppose (\ref{case:empty}) holds and $B$ is an RDP.
Then by \cite{Matsumoto:k3mun}*{Proposition 4.7},
there is $F \in k[[x,y,z^p]]$ such that 
$B \cong k[[x,y,z]]/(F)$, $D = \partialdd{z}$ up to unit multiple, and $B' = k[[x,y,z^p]] / (F)$,
and moreover all possible $F$ are classified.
The result is as in Table \ref{table:derivations on RDPs:non-fixed}.
\item \label{item:empty' and B' RDP} 
Dually, if (\ref{case:empty}$'$) holds and $B'$ is an RDP,
then we have a classification,
as in Table \ref{table:derivations on RDPs:non-fixed2}.
\item \label{item:empty and B RDP then nonempty'}
If (\ref{case:empty}) holds and $B$ is an RDP, then (\ref{case:point}$'$) holds.
This follows from the classification of assertion (\ref{item:empty and B RDP}).
\item \label{item:empty' and B' RDP then nonempty}
Dually, if (\ref{case:empty}$'$) holds and $B$ is an RDP, then (\ref{case:point}) holds.
\end{enumerate}

Among $2^4$ possibilities 
((\ref{case:empty}) or (\ref{case:point}), (\ref{case:empty}$'$) or (\ref{case:point}$'$), $B$ smooth or an RDP, $B'$ smooth or an RDP)),
$10$ are proved impossible by 
(\ref{item:B' smooth then empty}),
(\ref{item:B smooth then empty'}),
(\ref{item:empty and B RDP then nonempty'}), and
(\ref{item:empty' and B' RDP then nonempty}).
The case where $B$ and $B'$ are both smooth (and (\ref{case:empty}) and (\ref{case:empty}$'$) hold)) is done by the works cited in the statement of Theorem \ref{thm:main}(\ref{case:smooth}).
If (\ref{case:empty}) and (\ref{case:point}$'$) hold and $B$ is an RDP ($B'$ is either smooth or an RDP), then we have a classification by (\ref{item:empty and B RDP}).
Same for the dual cases.
Thus the only remaining case is where 
(\ref{case:point}) and (\ref{case:point}$'$) hold and $B$ and $B'$ are both RDPs.
Hereafter we consider this case.

If $D$ is of multiplicative type (i.e.\ $h = 1$)
then such derivations correspond to $\mu_p$-actions and are classified in \cite{Matsumoto:k3mun}*{Proposition 4.9}.
More generally, if $h \in B^*$, then we may replace $D$ with $h^{-1/(p-1)} D$,
which is of multiplicative type (since $D(h) = 0$ by Lemma \ref{lem:Dh=0}).
Dually, we have a classification for the cases where $h' \in B'^*$.
Hereafter we assume $h \in \idealm$ and $h' \in \idealm'$.

By \cite{Matsumoto:k3alphap}*{Corollary 3.5}, 
there are partial resolutions $\map{f}{\tilde{X}}{X}$ and $\map{f'}{\tilde{X'}}{X'}$ 
with a morphism $\map{\tilde{\pi}}{\tilde{X}}{\tilde{X}'}$
as in the diagram 
\[
\begin{tikzcd}
\tilde{X} \arrow[r,"\tilde{\pi}"] \arrow[d,"f"] & \tilde{X'} \arrow[d,"f'"] \\
X \arrow[r,"\pi"] & X' ,
\end{tikzcd}
\]
and a derivation $\tilde{D}$ on $\tilde{X}$ 
with $\tilde{D} = D$ outside the exceptional loci, 
satisfying the following properties: 
$\tilde{X'} = \tilde{X}^{\tilde{D}}$,
$\divisorialfix{\tilde{D}} = 0$, 
$\Sing(\tilde{X}) \cap \tilde{\pi}^{-1}(\Sing(\tilde{X}')) = \emptyset$
(in particular, $\Sing(\tilde{X}) \cap \Fix(\tilde{D}) = \emptyset$),
and $\Fix(\tilde{D}) \neq \emptyset$. 
We will refer to the points of $\Fix(\tilde{D})$ as the \emph{upper fixed points} of $D$.

If $h \in \cO_{\tilde{X}, w}^*$ at some upper fixed point $w \in \Fix(\tilde{D})$,
then $h \in B \cap \cO_{\tilde{X}, w}^* = B^*$,
contradicting our assumption that $h \in \idealm$.
Thus we have $h \in \idealm_w$ for any $w \in \Fix(\tilde{D})$.
Since $B_w$ is smooth, we have the classification \cite{Matsumoto:k3alphap}*{Lemma 3.6(2)}
of RDPs $B'_w$ satisfying this.
Suppose $p \geq 3$.
Then there are very few possibilities:
$p = 3$ and $B'_w$ is either of type $E_6^0$ or $E_8^0$, 
or $p = 5$ and $B'_w$ is of type $E_8^0$.
In other words, $B'$ admits one of these RDPs as a partial resolution.
Since no Dynkin diagram contains $E_8$ strictly,
the only possibility is that $p = 3$ and $\Sing(B') \in \set{E_7^r, E_8^r}$.
But $E_7^1$ and $E_8^2$ are impossible 
since their partial resolutions are $E_6^1$, not $E_6^0$, by \cite{Matsumoto:k3rdpht}*{Lemma 4.9}. 
Also $E_7^0$ is impossible, since by replacing $B$ and $B'$ with their universal coverings 
we obtain a derivation with satisfying (\ref{case:empty}) or (\ref{case:point}) and with quotient of type $E_6^0$, which is impossible by above.
Hence $\Sing(B') \in \set{E_8^0, E_8^1}$. 
Applying the same argument to $D'$, we obtain the same conclusion for $\Sing(B)$.
Thus we are in the case to be discussed in Section \ref{subsec:3E8}.

In the remaining case we have:
$p = 2$,
(\ref{case:point}) and (\ref{case:point}$'$) hold, 
$B$ and $B'$ are both RDPs, 
$h \in \idealm$, and $h' \in \idealm'$.
Thus we are in the case to be discussed in Section \ref{subsec:2}.

Before ending this section, we show the following.

\begin{lem} \label{lem:delta}
	Suppose $\map{\pi}{\Spec B}{\Spec B'}$ is as in Theorem \ref{thm:main} and satisfies (\ref{case:point}).
	Let $\delta = \dim_k \Image (\idealm'/\idealm'^{2} \to \idealm/\idealm^{2})$.
	Then $\delta \in \set{0,1}$.
\end{lem}
\begin{proof}
If $B$ is smooth, then $\delta < \dim \idealm/\idealm^2 = 2$ since $\idealm'$ cannot generate $\idealm$.

Suppose $B$ is an RDP.
	We have $\delta < \dim \idealm/\idealm^2 = 3$.
	Assume $\delta = 2$.
	We may assume $\idealm = (x, y, z)$ and $\idealm' = (x, y, z^p)$.
	If the ideal $(F) = \Ker (k[[x,y,z]] \to B)$ does not have a generator that belongs to $k[[x,y,z^p]]$, 
	then $B' = k[[x,y,z^p]] / (F^p)$,
	but this cannot be normal.
	Hence $(F)$ is generated by an element of $k[[x,y,z^p]]$, which we may assume to be $F$ itself,
	and $B' = k[[x,y,z^p]] / (F)$.
	Then $D$ is proportional to $\partialdd{z}$, contradicting the assumption (\ref{case:point}). 
\end{proof}
It turns out that $\delta$ is always $1$ if $B$ is an RDP, but we can prove this only after case-by-case arguments.

\subsection{Case of $E_8$ in $p = 3$} \label{subsec:3E8}

Suppose $p = 3$, both $B$ and $B'$ are either of type $E_8^0$ or $E_8^1$, and
both $h$ and $h'$ are non-units at every upper fixed point of $D$.
By the arguments in Section \ref{subsec:reduction},
there is exactly one upper fixed point of $D$, and its quotient is of type $E_6^0$.
Let $X_1 \to X$ be the blow-up at the closed point.
Then $\Sing(X_1) = E_7^0$.
Since there is no derivation on $E_7^0$ fixing precisely the closed point with non-unit $h$ and RDP quotient,
this point is not fixed, and the upper fixed point should lie above $X_1^{\sm}$.
In particular, there exists a unique fixed point on $X_1^{\sm}$.

Suppose $B$ is of type $E_8^1$.
We may assume that $B = k[[x,y,z]] / (z^2 + x^3 + y^5 + x^2 y^3)$.
Then the space of derivations on $B$ are generated by the following elements $D_1, D_2, D_3$.
Then $D$ has at least two fixed points $(x = 1 + y_1^2 = z_1 = 0)$ 
on $\Spec B[y/x,z/x] = \Spec k[[x]][y_1,z_1] / (z_1^2 + x + x^3 y_1^3 (y_1^2 + 1)) \subset X_1$, contradiction.

Dually, $B'$ is not of type $E_8^1$.

\begin{tabular}{llll}
\toprule
            & $D_1$  & $D_2$   & $D_3$ \\
\midrule
$x$         & $0$    & $z$     & $y$  \\
$y$         & $z$    & $0$     & $-x$ \\
$z$         & $-y^4$ & $-xy^3$ & $0$  \\
\midrule
$x$         & $0$          & $x z_1$            & $x y_1$   \\
$y_1 = y/x$ & $z_1$        & $- y_1 z_1$        & $-1 - y_1^2$ \\
$z_1 = z/x$ & $-x^3 y_1^4$ & $-x^3 y_1^3 - z_1^2$ & $- y_1 z_1$ \\
\bottomrule
\end{tabular}

\medskip

Now suppose $B$ and $B'$ are of type $E_8^0$.
\begin{claim}
$D \restrictedto{\idealm/\idealm^2}$ is nilpotent of index $3$
and $\delta = 1$, where $\delta$ is as in Lemma \ref{lem:delta}.
\end{claim}
\begin{proof}[Proof of Claim]
	Since we know that $\delta \leq 1$ by Lemma \ref{lem:delta}
	and that $(D \restrictedto{\idealm/\idealm^2})^3 = 0$ since $h \in \idealm$,
	it remains to show that $\delta \geq 1$ and $(D \restrictedto{\idealm/\idealm^2})^2 \neq 0$.

	We write $B = \Spec k[[x,y,z]] / (z^2 + x^3 + y^5)$
	during this proof.
	Then the space of derivations on $B$ are generated by the following elements $D_1, D_2$.
	
	\begin{tabular}{lll}
		\toprule
		& $D_1$ & $D_2$ \\
		\midrule
		$x$ & $0$    & $1$ \\
		$y$ & $z$    & $0$ \\
		$z$ & $-y^4$ & $0$ \\
		\midrule
		$x$         & $0$           & $1$ \\
		$y_1 = y/x$ & $z_1$         & $- y_1/x$ \\
		$z_1 = z/x$ & $- x^3 y_1^4$ & $- z_1/x$ \\
		\midrule
		$x_2 = x/y$ & $- x_2 z_2$     & $1/y$ \\
		$y$         & $y z_2$         & $0$ \\
		$z_2 = z/y$ & $- y^3 - z_2^2$ & $0$ \\
		\bottomrule
	\end{tabular}
	
	Write $D = f_1 D_1 + f_2 D_2$ with $f_i \in B$,
	and suppose $f_i \equiv f_{i0} + f_{i1} x + f_{i2} y + f_{i3} z \pmod{\idealm^2}$ with $f_{ij} \in k$.
	Since $D$ extends to $X_1$, we have $f_2 \in \idealm$ ($f_{20} = 0$).
	Since $h \not\in B^*$, we have $f_{21} = 0$.
	Since $D$ does not fix the origin of 
	$\Spec B[x/y, z/y] = \Spec k[[y]][x_2,z_2] / (z_2^2 + x_2^3 y + y^3) \subset X_1 = \Bl_{\idealm} X$,
	we have $f_{22} \neq 0$.
	Let $\idealm_1 = (y_1, z_1)$ be the maximal ideal at the origin of 
	$\Spec B[y/x, z/x] = \Spec k[[x]][y_1,z_1] / (z_1^2 + x + x^3 y_1^5) \subset X_1$,
	which is a fixed point of $D$.
	Since the image of this point is of type $E_6^0$, we have $D(\idealm_1) \not\subset \idealm_1^2$,
	hence $f_{10} \neq 0$. 
	Then $v := D(D(x)) - hx$ satisfies
	$D(v) = (D^3 - h D)(x) = 0$ (by Lemma \ref{lem:Dh=0}), hence $v \in \Image (\idealm' \to \idealm)$, and 
	$v \equiv D(D(x)) \equiv D(f_{22} y + f_{23} z) \equiv f_{10} f_{22} z \not\equiv 0 \pmod{\idealm^2}$.
\end{proof}

We note that if $(B = k[[x,y,z]]/(F),\idealm)$ is an RDP of type $D_n$ or $E_n$ in characteristic $\geq 3$ and
$D$ is a derivation on $B$ with $D(\idealm) \subset \idealm$ and $D \restrictedto{\idealm/\idealm^2}$ is nilpotent of index $3$,
then $F \equiv l^2 \pmod{\idealm^3}$ with $l \in \Ker(D \restrictedto{\idealm/\idealm^2})$
(otherwise $D(F)$ cannot be zero).

\medskip

We will show that $B$ admits elements $x,y,z,w$ as in the statement of Theorem \ref{thm:main}(\ref{case:both fixed:principal}).
We may assume $\idealm = (x,y,z)$, $z \in B'$,
and $\idealm' = (Y, z, w)$, where $Y := y^3$.
We may moreover assume $x^3 \in \idealm'^2$ and $w \in \idealm^2$.
Write $w = Q(z,y,x) \in (z,y,x)^2 \subset k[[z,y,x]]$.
Since $x^3 \in \idealm'^2$, there exists $u \in k[[x,y,z]]$ and $P \in (Y,z,w)^2 \subset k[[Y,z,w]]$ 
such that $u F = - x^3 + P(Y, z, Q(z,y,x))$.
If $u \not\in k[[x,y,z]]^*$,
then the degree $2$ part of $F$ cannot be the square of a linear term contained in 
$\Ker(D \restrictedto{\idealm/\idealm^2}) = k z$, contradicting the observation above.
Hence $u \in k[[x,y,z]]^*$, and we may assume $u = 1$: $F = - x^3 + P(Y, z, Q(z,y,x))$,
and also $F' = -w^3 + \tilde{Q}(z^3, Y, P(Y,z,w))$,
where $\tilde{Q}(z^3,Y,P) := Q(z,y,x)^3$. 

We will say that a polynomial \emph{has} a monomial 
if the corresponding coefficient is nonzero.

Since $B$ and $B'$ are of type $E_8$, we obtain the following:
$P$ has $z^2$; $Q$ has $y^2$;
we may assume $P$ does not have $zw$ nor $w^2$;
we may assume $Q$ does not have $yx$ nor $x^2$;
$P$ has $Yw$; $Q$ has $zx$.
By replacing $x,y,z,w,F,F'$ with scalar multiples,
we may assume $F = -x^3 + z^2 + y^3 w + (\dots)$, $F' = (-w + y^2 + z x + (\dots))^3$.
Since $0 = D(w) = D(x) Q_x + D(y) Q_y$
and since $Q_x \equiv z \pmod{\idealm^2}$
and $Q_y \equiv 2 y \pmod {(z) + \idealm^2}$ have no common factor,
we may assume $(D(x), D(y)) = (- Q_y, Q_x)$.

\subsection{Case of characteristic $2$} \label{subsec:2}

Suppose $p = 2$ and 
both $h$ and $h'$ are non-units. 

In this case, there are still many possibilities for the quotients of upper fixed  points, hence the geometric arguments used in Section \ref{subsec:3E8} are not effective.
Instead, we will determine the equations directly (up to terms of high degree).

As in Lemma \ref{lem:delta},
we let $\delta = \dim_k \Image (\idealm'/\idealm'^{2} \to \idealm/\idealm^{2})$, 
and similarly $\delta' = \dim_k \Image (\secondpower{\idealm}/(\secondpower{\idealm})^{2} \to \idealm'/\idealm'^{2})$.
We have $\delta, \delta' \in \set{0,1}$.
We shall show $\delta = \delta' = 1$.

	Assume $\delta = 0$, that is, $\idealm' \subset \idealm^{2}$.
	Since $x^2,y^2,z^2 \in \secondpower{\Frac B} \subsetneq \Frac B'$, 
	the elements $x^2,y^2,z^2 \subset \secondpower{\idealm} \subset \idealm'$ cannot generate $\idealm'$.
	In other words, there exists a nonzero linear combination $f$ of $x^2,y^2,z^2$ that belongs to $\idealm'^2$.
	Since $f \in \idealm'^2 \subset \idealm^4$,
	(a unit multiple of) $F$ is of the form $f + G$ with $f \in \secondpower{\idealm}$ and $G \in \idealm^4$.
	But such a polynomial cannot define an RDP.
Hence $\delta = 1$. Dually $\delta' = 1$.

	We may assume that $\idealm = (x, y, z)$ and $\idealm' = (w, y^2, z)$, 
	and that $x^2 \in \idealm'^2$ and $w^2 \in (\secondpower{\idealm})^2$.

	Since $x^2 \in \idealm'^2$, we may assume $F = x^2 + P_0(y^2, z) + P_1(y^2, z) w$
	with $P_0 \in \idealn'^2$ and $P_1 \in \idealn'$,
	where $\idealn'$ is the maximal ideal $(y^2, z)$ of $k[[y^2, z]]$.
	Similarly we may assume $w = Q_0(z, y) + Q_1(z, y) x$ with $Q_0 \in \idealn^2$ and $Q_1 \in \idealn$, 
	where $\idealn$ is the maximal ideal $(z, y)$ of $k[[z, y]]$.
	Hence we have 
	\begin{align*}
		F  &= x^2 + P_0(y^2, z) + P_1(y^2, z) (Q_0(z, y) + Q_1(z, y) x), \\
		F' &= w^2 + \tilde{Q}_0(z^2, Y) + \tilde{Q}_1(z^2, Y) (P_0(Y, z) + P_1(Y,z) w),
	\end{align*}
	where $Y := y^2$, and $\tilde{Q}_i(z^2, Y) := Q_i(z, y)^2$. 
	
	As in Section \ref{subsec:3E8}, 
	we will say that a polynomial \emph{has} a monomial 
if the corresponding coefficient is nonzero.

	We may assume $P_0$ does not have $z^2$, and $\tilde{Q}_0$ does not have $Y^2$. 
	
	We show that $Q_1$ does not have $y$.
	For this, consider the derivation $D_1$ defined by $D_1(x,y,z) = ((Q_0)_y + (Q_1)_y x, Q_1, 0)$.
	Then it satisfies $B^{D_1} = B'$, hence $D_1 = a D$ for some $a \in B \setminus \set{0}$, and $D_1^2 = (Q_1)_y D_1$.
	By Corollary \ref{cor:h unit} we obtain $(Q_1)_y \in \idealm$.
	Similarly, $P_1$ does not have $z$.

	For $F$ to define an RDP, we need either 
	\begin{itemize}
	\item $P_0$ has $y^2 z$ (then we may assume $P_0$ does not have $z^3$), $P_1$ has $z^k$, and either $Q_0$ has $z^m y$ or $Q_1$ has $z^l$, or
	\item $P_0$ does not have $y^2 z$ but has $z^3$, $P_1$ has $y^2$, and $Q_0$ has $zy$ or $y^3$.
	\end{itemize}
	Similarly, for $F'$ to define an RDP, we need either
	\begin{itemize}
		\item $\tilde{Q}_0$ has $z^2 Y$ (then we may assume $\tilde{Q}_0$ does not have $Y^3$), $\tilde{Q}_1$ has $Y^{k'}$, and either $P_0$ has $Y^{m'} z$ or $P_1$ has $Y^{l'}$, or
		\item $\tilde{Q}_0$ does not have $z^2 Y$ but has $Y^3$, $\tilde{Q}_1$ has $z^2$, and $P_0$ has $Yz$ or $z^3$.
	\end{itemize}
	Combining these conditions, one of the following holds.
	\begin{itemize}
		\item $P_0$ has $y^2 z$ and does not have $z^3$, $P_1$ has $z^{m'}$, $Q_0$ has $zy$ and does not have $y^3$, and $Q_1$ has $y^{m}$.
		In this case, $B$ is of type $D_{2m'+2}^{\positive{2-m}}$
		and $B'$ is of type $D_{2m+2}^{\positive{2-m'}}$.
		\item $P_0$ does not have $y^2 z$ but has $z^3$, $P_1$ has $y^2$, $Q_0$ has $zy$ and does not have $y^3$, and $Q_1$ has $y^m$.
		In this case, $B$ is of type $E_7^{\positive{3-m}}$ and $B'$ is of type $D_{2m+3}^0$.
		\item Dual of the previous case:
		$B$ is of type $D_{2m+3}^0$ and $B'$ is of type $E_7^{\positive{3-m}}$.
		\item $P_0$ does not have $y^2 z$ but has $z^3$, $P_1$ has $y^2$, 
		$Q_0$ does not have $zy$ but has $y^3$, and $Q_1$ has $z$.
		In this case, $B$ and $B'$ are of type $E_8^2$.
	\end{itemize}

\subsection{Non-Cartier case} \label{subsec:non-principal}

Suppose the divisor $\divisorialfix{D}$ is not Cartier.
Let $l$ be the order of the class $[\divisorialfix{D}]$ in $\Cl(B)$. 
We have $l \divides (p-1)$ by the Rudakov--Shafarevich formula (Proposition \ref{prop:2-forms}).
Let $\map{f}{\bar{X} = \Spec \bar{B}}{X = \Spec B}$ 
be the covering (\'etale outside the closed point) trivializing the class,
and let $\bar{\idealm} \subset \bar{B}$ be the maximal ideal.
Let $\bar{D}$ be a derivation proportional to $f^* D$ with $\divisorialfix{\bar{D}} = 0$ 
(which exists since $[\divisorialfix{f^* D}] = 0$).
Here the pullback $f^* D$ is first defined on the \'etale locus of $f$, and then extended to $\bar{X}$ since $f$ is \'etale in codimension $1$.
Let $\bar{X}' = \Spec \bar{B}' = \bar{X}^{\bar{D}}$.

Let $G = \Gal(\bar{X}/X) = \Gal(\bar{X'}/X') \cong \bZ/l\bZ$.
For $g \in G$, the derivation $g^* \bar{D}$ 
can be written as $g^* \bar{D} := g^{\circ} \bar{D} (g^{\circ})^{-1}$,
where $\map{g^{\circ}}{\bar{B}}{\bar{B}}$ corresponds (contravariantly) to $\map{g}{\bar{X}}{\bar{X}}$.
Since $\bar{B}^{g^* \bar{D}} = \bar{B}^{\bar{D}}$ and $\divisorialfix{g^* \bar{D}} = 0 = \divisorialfix{\bar{D}}$, 
there exists $\beta_g \in \bar{B}^*$ such that $g^* \bar{D} = \beta_g \cdot \bar{D}$.
We have $\beta_{hg} = g^{\circ}(\beta_h) \beta_g$.
This shows that ($\beta$ is a $1$-cocycle and),
since $G \actson \bar{B}/\bar{\idealm} = k$ is trivial,
the map $\rho \colon G \namedto{\beta} \bar{B}^* \to (\bar{B}/\bar{\idealm})^* = k^*$ is a homomorphism.

We will show that this $\map{\rho}{G}{k^*}$ is injective.
Let $G_1 := \Ker \rho$. 
Let $\bar{D}_1 := \sum_{h \in G_1} h^* \bar{D} = (\sum_{h \in G_1} \beta_h) \cdot \bar{D}$.
Then $\bar{D}_1$ is $G_1$-invariant,
hence descends to a derivation on $\bar{B}^{G_1}$.
Since $\sum_{h \in G_1} \beta_h \in \bar{B}^*$ 
(since $\sum_{h \in G_1}  \beta_h \equiv \sum 1 = \abs{G_1} \not\equiv 0 \pmod{\bar{\idealm}}$),
$\bar{D}_1$ has no divisorial fixed locus.
In other words, the pullback of the class $[\divisorialfix{D}]$ to $\bar{X}/G_1$ is trivial.
By the definition of $\bar{X}$, we obtain $G_1 = \set{1}$.

Summarizing the arguments so far, 
the at most RDP $\bar{X} = \Spec \bar{B}$, 
the action of the cyclic group $G$, and the derivation $\bar{D}$ satisfy the following properties.
\begin{itemize}
	\item $\bar{X} \to \bar{X}/G = X$ is induced by a non-trivial divisor class
	of order dividing $p-1$ on an at most RDP $X$.
	Such $\bar{X} \to X$ can be listed from the list of RDPs and their class groups.
	\item $\bar{D}$ is $p$-closed with $\divisorialfix{\bar{D}} = 0$, and $\bar{X}^{\bar{D}}$ is at most RDP.
	Such $\bar{X} \to \bar{X}^{\bar{D}}$ is classified in the cases (\ref{case:empty}) and (\ref{case:point}) of Theorem \ref{thm:main}.
	\item For each $g \in G$ we have $g^* \bar{D} = \beta_g \cdot \bar{D}$ for some unit $\beta_g \in \bar{B}^*$,
	and the homomorphism $\map{\rho}{G}{k^*} \colon {g} \mapsto {(\beta_g \bmod \bar{\idealm})}$ is injective.
\end{itemize}

The first two conditions imply that $(\Sing(B), \Sing(\bar{B}), \Sing(\bar{B}^{\bar{D}}))$ 
is one in Table \ref{table:covering} (with $n \geq 1$),
that $(\bar{B}, G) \cong (k[[x,y,z]] / \bar{F}, \spanned{g})$ as in the table,
and that $(\bar{B}, \bar{D}) \cong (k[[x,y,z]] / (\bar{F}), u \bar{D})$ as in the table
with $u \in \bar{B}^*$,
but not necessarily under the same isomorphism $\bar{B} \cong k[[x,y,z]]/(\bar{F})$. 
We show that:
\begin{itemize}
\item \label{item:4 realizable cases}
Among the candidates in Table \ref{table:covering},
those without check marks are inappropriate.
\item \label{item:description}
In the other cases (with check marks),
we have an isomorphism $\bar{B} = k[[x,y,z]]/(\bar{F})$ 
with $\bar{F}$ and the both actions of $\bar{D}$ and $G$ as in the table,
up to replacing $\bar{D}$ by a unit multiple.
\end{itemize}
Then it is straightforward 
to compute $B = \bar{B}^G$
and the restrictions to $B$ of suitable multiples of $\bar{D}$,
which gives the formulas as in Table \ref{table:derivations on RDPs:non-principal}.

First consider the first $4$ cases of Table \ref{table:covering},
where $\bar{D}$ satisfies (\ref{case:empty}).
Let $V = \bar{\idealm} / \bar{\idealm}^2$ and $V' = \Image(\bar{\idealm}' \to V)$.
These spaces are stable under $G$, hence $G$ acts on $V/V'$, which is $1$-dimensional by the classification.
The linear map ${V/V'} \to \bar{B}/\bar{\idealm} = k$ induced by $\bar{D}$ is nonzero (since (\ref{case:empty})),
and is $G$-equivariant where the action on $k$ is $\rho^{-1}$.

In the case where $\bar{B}$ is smooth,
a generator $g$ of $G$ acts on $V/V'$ by a primitive $l$-th root $\zeta$ of $1$ and on $V'$ by $\zeta^{-1}$.
We obtain lifts $y \in \bar{\idealm}'$ and $x \in \bar{\idealm}$ with $g(y) = \zeta^{-1} y$ and $g(x) = \zeta x$.

In the other $3$ cases where $\bar{B}$ is of type $A_{np-1}$ there is,
since $np \geq p \geq 3$, a canonical $2$-element set $\set{kx, ky} \subset \bP(V)$ 
characterized by the property that $xy \in \bar{\idealm}^3$.
Then the order of the action of $G$ on $V/V'$ is $1,2,2$ respectively, by the classification.
Since this should be equal to $l = l,2,4$, the first and the third cases are inappropriate.
Consider the case where $\bar{B}$ is of type $A_{np-1}$ and $B$ is of type $D_{np/2+2}$, with $n$ even.
We sketch how to find $x,y,z$.
We may assume the nontrivial element $g \in G$ interchanges $x$ and $y$.
We take $z$ such that $x,y,z$ generate $\bar{\idealm}$ and $g(z) = -z$.
Since $\bar{B}$ is of type $A_{np-1}$, we have $xy = c z^{np}$ ($c \in k^*$) up to terms of high degree.
By a standard argument we can modify $x,y,z$ so that the actions of $G$ and $\bar{D}$ are as above and they satisfy $xy = z^{np}$. We omit the details.

The next $3$ cases of Table \ref{table:covering},
where $\bar{B}$ is of type $A_{n-1}$ and $\bar{B}'$ is of type $A_{pn-1}$, are simply the dual of the previous cases.

Finally, consider the case $\bar{B}$ is of type $E_6^1$ in characteristic $3$.
As before, let $V = \bar{\idealm} / \bar{\idealm}^2$.
The line $kz \in \bP(V)$ is characterized by $z^2 \in \bar{\idealm}^3$,
and we can take a lift $z \in \bar{\idealm}$ satisfying both $g(z) = -z$ and $D(z) = 0$.
We take $x,y \in \bar{\idealm}$ such that $x,y,z$ generate $\bar{\idealm}$, $g(x,y) = (y,x)$,
and $D(x,y) = (ix, -iy)$.
Since $\bar{B}$ is of type $E_6^1$, we have $z^2 + c (x^3 + y^3) + c' x^2 y^2 + (\text{higher}) = 0$ for some $c,c' \in k^*$.
We may assume $c = c' = 1$.
Again we omit the details to modify $x,y,z$.

\begin{table}
	\caption{Candidates for $(B, \bar{B}, \bar{B}^{\bar{D}})$}
	\label{table:covering}
	{ \scriptsize
		\begin{tabular}{lllllllll}
			\toprule
			& $p$ & $l$ & $B$ & $\bar{B}$ & $\bar{B}^{\bar{D}}$ & $\bar{F}$ & $\bar{D}$ & $g$ \\
			\midrule
			\checkmark & any & $l$ & $A_{l-1}$                & smooth    & smooth     & & $0, 1$ & $\zeta x, \zeta^{-1} y$ \\
			           & any & $l$ & $A_{lnp-1}$ & $A_{np-1}$ & $A_{n-1}$ & $xy - z^{np}$ & $0, 0, 1$ & $\zeta x, \zeta^{-1} y, z$ \\
			\checkmark & any & $2$ & $D_{np/2+2}$ ($n$ even) & $A_{np-1}$ & $A_{n-1}$ & $xy - z^{np}$ & $0, 0, 1$ & $y, x, -z$ \\
			           & any & $4$ & $D_{np+2}$ ($n$ odd) & $A_{np-1}$ & $A_{n-1}$ & $xy - z^{np}$ & $0, 0, 1$ & $y, -x, -z$ \\
			           & any & $l$ & $A_{ln-1}$  & $A_{n-1}$ & $A_{np-1}$ & $xy - z^n$ & $x, -y, 0$ & $\zeta x, \zeta^{-1} y, z$ \\
			\checkmark & any & $2$ & $D_{n/2+2}$ ($n$ even) & $A_{n-1}$ & $A_{np-1}$ & $xy - z^{n}$ & $x, -y, 0$ & $y, x, -z$ \\
			           & any & $4$ & $D_{n+2}$ ($n$ odd) & $A_{n-1}$ & $A_{np-1}$ & $xy - z^{n}$ & $x, -y, 0$ & $y, -x, -z$ \\
			\checkmark & $3$ & $2$ & $E_7^1$ & $E_6^1$ & $E_6^1$ & $z^2 + x^3 + y^3 + x^2 y^2$ & $x, -y, 0$ & $y, x, -z$ \\
			\bottomrule
		\end{tabular}
	}
\end{table}

\section{Smooth coverings of RDPs} \label{sec:smooth coverings}

In this section we construct smooth coverings of a certain type for every RDPs (Theorem \ref{thm:smooth covering}), 
which refines the result of Artin.
We also prove some variants (Corollary \ref{cor:isomorphic pi1} and Proposition \ref{prop:smooth covering etale last}).

\begin{thm} \label{thm:smooth covering}
	Suppose $B$ is a complete local RDP in characteristic $p > 0$.
	Then there exists
	a sequence $B \subset C_0 \subset C_1 \subset \dots \subset C_{n}$ ($n \geq 0$)
	of finite extensions of complete local algebras such that
	\begin{itemize}
		\item $C_{n}$ is smooth, and all other $C_i$ are RDPs,
		\item $\Spec C_0 \to \Spec B$ is the universal covering, and 
		\item $\Spec C_{i+1} \to \Spec C_{i}$ are purely inseparable of degree $p$, of type (\ref{case:point}),
	\end{itemize}
	if and only if $(p, \Sing(B)) \neq (2, E_8^1)$.
	If $B$ is of type $E_8^1$ in characteristic $2$, then we only have a weaker assertion where $C_i$ are not assumed to be RDPs.
\end{thm}

\begin{proof}
	First suppose $B$ is not of type $E_8^1$ in characteristic $2$.
	As mentioned in Convention \ref{conv:RDP}(\ref{conv:RDP:universal covering}),
	the universal covering of an at most RDP is again an at most RDP.
	Hence it suffices to consider the case $B$ is simply-connected.
	See Table \ref{table:pi1} (due to \cite{Artin:RDP}) for the fundamental groups and the universal coverings of RDPs.
	By Theorem \ref{thm:main} we have the following sequences.

	$A_{p^e-1} \from A_{p^{e-1}-1} \from \dots \from \smpoint$.
	
	For $p = 5$: $E_8^0 \from \smpoint$.
	
	For $p = 3$: $E_8^0 \from \smpoint$, $E_8^1 \from E_6^0 \from \smpoint$.
	
	For $p = 2$: 
	$E_8^3 \from E_7^2 \from D_5^{1/2} \from A_1 \from \smpoint$,
	$E_8^0 \from \smpoint$,
	$E_7^1 \from D_7^{1/2} \from \smpoint$,
	$E_7^0 \from \smpoint$,
	$D_{2m}^0 \from \smpoint$,
	$D_{2m+1}^{1/2} \from A_1 \from \smpoint$,
	$D_{4k-l}^{k-l/2} \from D_{4k-2l}^{\positive{k-l}} \from \dots \from 
	D_{4k-2^n l}^{0} \from \smpoint$ if $l > 0$ and $k - l/2 > 1/2$,
	where $n$ is the minimum integer with $4k - 2^{n-1} l \leq 0$.
	
	Note that the correspondence $(r, N) = (k-l/2, 4k-l)$ gives a bijection between the sets
	$\set{(r, N) \in \frac{1}{2} \bZ \times \bZ \mid N \geq 4, 2r - N \in \bZ, 0 \leq r \leq (N/2)-1}$ and
	$\set{(k, l) \in \bZ \times \bZ \mid k \geq 1, k \geq l/2, (k,l) \neq (1,1),(1,2)}$,
	and $4r < N$ if and only if $l > 0$.

\bigskip

	Now suppose $B$ is of type $E_8^1$ in characteristic $2$.
	Then $C_0 = B$ is not smooth and, 
	by Theorem \ref{thm:main},
	it does not admit any purely inseparable covering of degree $p$ that is at most RDP.
	Hence the assertion with the RDP assumption does not hold.
To prove the weaker assertion (without the RDP assumption),
we will give a sequence $B = C_0 \subset C_1 \subset C_2$
of purely inseparable coverings of degree $p$ of type (\ref{case:point})
such that ($C_1$ is a non-RDP and) $C_2$ is an RDP of another type.
Take $c \in k$ and let 
\begin{align*}
        C_2 &= k[[t,v,y]] / (v^2 (1 + t y^2) + (c + y) (y^4 v + t^2)) \\
\supset C_1 &= k[[x,y,z]] / (x^2 + z^3 + z y^2(y^2 (c + y) + z  x)), \\
\supset C_0 &= k[[z,Y,w]] / (w^2 + Y^3 + z^2 (z^3 + z Y w + c z Y^2)), 
\end{align*}
where 
\[
x = \frac{y^2 v + t^3}{1 + t y^2}, \quad 
z = \frac{y^4 v + t^2}{1 + t y^2}, \quad 
w = y^3 + x z, \quad 
Y = y^2.
\]
Define derivations $\map{\delta_1}{C_2}{C_2}$ and $\map{\delta'_2}{C_1}{C_1}$ by
\[
\delta_1(t, v, y) = (y^2, z, 0), \qquad
\delta'_2(x, y, z) = (y^2, z, 0).
\]
Then it is straightforward to check that 
$C_2$ is an RDP of type $D_8^0$ or $D_{11}^{1/2}$ if $c \neq 0$ or $c = 0$ respectively,
$C_0$ is an RDP of type $E_8^1$,
both $\delta_1$ and $\delta'_2$ are of additive type, 
both $\Fix(\delta_1)$ and $\Fix(\delta'_2)$ consist precisely of the closed point,
$C_1 = C_2^{\delta_1}$, and $C_0 = C_1^{\delta'_2}$.

For the (non-RDP) singularity of $C_1$, see Remark \ref{rem:covering of E_8^1}.
\end{proof}

\begin{cor} \label{cor:isomorphic pi1}
	Suppose $B_1$ and $B_2$ are at most RDPs (over the same algebraically closed field $k$) in characteristic $p > 0$
	having isomorphic local fundamental groups $\pietloc(B_1) \cong \pietloc(B_2)$.
	If $(p, \pietloc) \neq (2, \Dih_n)$ for any $n \geq 1$,
	then $B_1$ and $B_2$ are connected by a finite purely inseparable morphism.
\end{cor}
Here $\Dih_n$ is the dihedral group (of order $2n$). In particular, $\Dih_1 = \C_2$.
Again, see Table \ref{table:pi1} (due to \cite{Artin:RDP}) for the fundamental groups and the universal coverings of RDPs.

\begin{table} 
	\caption{Etale fundamental groups and the universal coverings of RDPs} \label{table:pi1}
\begin{tabular}{lllll}
\toprule
char & univ.\ cov. & RDP & & $\pi_1$ \\
\midrule
any          & $A_{p^e-1}$ & $A_{n p^e - 1}$ & ($p \notdivides n$) & $\C_{n}$: cyclic (of order $n$) \\
$\neq 2$     & $A_{p^e-1}$ & $D_{n p^e + 2}$ & ($p \notdivides n$) & $\BinDih_{n}$: binary dihedral (of order $4 n$) \\
$\neq 2,3$   & smooth      & $E_6$           &                     & $\BinTet$: binary tetrahedral (of order $24$) \\
$\neq 2,3$   & smooth      & $E_7$           &                     & $\BinOct$: binary octahedral (of order $48$) \\
$\neq 2,3,5$ & smooth      & $E_8$           &                     & $\BinIco$: binary icosahedral (of order $120$) \\
\midrule
$5$ & $E_8^0$         & $E_8^0$   &             & $0$ \\
$5$ & smooth          & $E_8^1$   &             & $\C_5$ \\
\midrule
$3$ & $E_6^0$         & $E_6^0$   &             & $0$ \\
$3$ & smooth          & $E_6^1$   &             & $\C_3$ \\
$3$ & $E_6^0$         & $E_7^0$   &             & $\C_2$ \\
$3$ & smooth          & $E_7^1$   &             & $\C_6$ \\
$3$ & $E_8^r$         & $E_8^r$   & ($r = 0,1$) & $0$ \\
$3$ & smooth          & $E_8^2$   &             & $\BinTet$: binary tetrahedral (of order $24$) \\
\midrule
$2$ & $A_{2^{e+1}-1}$ & $D_{N}^r$ & ($4r > N$)  & $\Dih_{(4r-N)'}$, $4r - N = 2^e (4r - N)'$, \\
&&&& \quad $2 \notdivides (4r-N)'$: \\
&&&& \quad dihedral (of order $2 \cdot (4r-N)'$) \\
$2$ & smooth          & $D_{N}^r$ & ($4r = N$)  & $\C_2$  \\
$2$ & $D_{N}^r$       & $D_{N}^r$ & ($4r < N$)  & $0$  \\
$2$ & $D_4^0$         & $E_6^0$   &             & $\C_3$ \\
$2$ & smooth          & $E_6^1$   &             & $\C_6$ \\
$2$ & $E_7^r$         & $E_7^r$   & ($r=0,1,2$) & $0$ \\
$2$ & smooth          & $E_7^3$   &             & $\C_4$ \\
$2$ & $E_8^r$         & $E_8^r$   & ($r=0,1,3$) & $0$ \\
$2$ & smooth          & $E_8^2$   &             & $\C_2$ \\
$2$ & smooth          & $E_8^4$   &             & metacyclic of order $12$ \\
\bottomrule
\end{tabular}
\end{table}

\begin{proof}
	If $\pietloc = 0$, then both $B_1$ and $B_2$ are connected to $k[[x,y]]$ 
	by Theorem \ref{thm:smooth covering}.

	Suppose $\pietloc \neq 0$.
	Theorem \ref{thm:main} shows that 
	$B_1$ and $B_2$ are connected by a purely inseparable morphism of degree $p$ 
	if $(B_1,B_2)$ is one of the following:
	$(A_{n p^e - 1}, A_{n p^{e'} - 1})$, 
	$(D_{n p^e + 2}, D_{n p^{e'} + 2})$ ($p \neq 2$), 
	$(A_1, E_7^0)$ ($p = 3$),
	$(A_2, E_6^0)$ ($p = 2$).
	According to Table \ref{table:pi1}, 
	general cases (with $(p, \pietloc) \neq (2, \Dih_n)$) follow from these cases.
\end{proof}

Also, in some (not all) cases the order of the \'etale and purely inseparable coverings in Theorem \ref{thm:smooth covering} can be reversed.
\begin{prop} \label{prop:smooth covering etale last}
	Suppose $B$ is as in Theorem \ref{thm:smooth covering}.
	Then there exists 
	a sequence $B = B_0 \subset B_1 \subset \dots \subset B_n \subset C$ ($n \geq 0$) 
	of finite extensions of complete local algebras such that
	\begin{itemize}
		\item $C$ is smooth, and all other $B_i$ are RDPs,
		\item $\Spec C \to \Spec B_n$ is the universal covering, and 
		\item $\Spec B_{i+1} \to \Spec B_{i}$ are purely inseparable of degree $p$, of type (\ref{case:point}) or (\ref{case:non-principal}),
	\end{itemize}
	if and only if $(p, \Sing(B)) \neq (2, E_8^1), (2, D_N^r)$ ($4r > N$).
\end{prop}

\begin{proof}
	Again, this is impossible for $E_8^1$ in characteristic $2$.
	
	Next, if $p = 2$ and $\Sing(B)$ belongs to the set $\set{D_N^r \mid 4r > N}$, 
	then by Theorem \ref{thm:main} so does any RDP that is a purely inseparable covering of $B$ of degree $p$,
	and no member of this family has smooth universal covering.
	
	Suppose $B$ is none of the above.
	We shall describe $\Spec B_0 \from \Spec B_1 \from \dots \from \Spec B_{n}$
	with $B_n$ having smooth universal covering.
	By Theorem \ref{thm:smooth covering}, it remains to consider the cases with $\pietloc(B) \neq 0$.
	
	$A_{n p^e-1} \from A_{n p^{e-1}-1} \from \dots \from A_{n-1}$ ($p \notdivides n$).
	
	$D_{n p^e + 2} \from D_{n p^{e-1} + 2} \from \dots \from D_{n+2}$ ($p \neq 2$, $p \notdivides n$).
	(Here, the corresponding derivations satisfy (\ref{case:non-principal}).)
	
	For $p = 3$: $E_7^0 \from A_1$.
	
	For $p = 2$: $E_6^0 \from A_2$.
\end{proof}

\begin{rem}
In fact, we can show that $\Spec C_2 \to \Spec C_0$ we gave in the case of $E_8^1$ of the proof of Theorem \ref{thm:smooth covering}
is the quotient by an action of the group scheme $\alpha_4$.
The existence of such a quotient morphism will be used in \cite{Liedtke--Martin--Matsumoto:RDPtors}*{Remark 9.32}.

Here $\alpha_{p^e} = \Spec R$, $R = k[\varepsilon] / (\varepsilon^{p^e})$, is the group scheme with comultiplication given by 
$R \to R \otimes R \colon \varepsilon \mapsto \varepsilon \otimes 1 + 1 \otimes \varepsilon$.
Just as $\alpha_p$-actions correspond to derivations of additive type,
$\alpha_{p^e}$-actions can be described by using so-called higher derivations. 
For simplicity, we consider only the case $p^e = 4$.
Giving an $\alpha_4$-action on $\Spec A$ is equivalent to giving two $k$-linear maps $\map{\delta_1, \delta_2}{A}{A}$ satisfying the following properties:
\begin{itemize}
\item $\delta_1^2 = \delta_2^2 = 0$ and $\delta_1 \delta_2 = \delta_2 \delta_1$,
\item $\delta_1$ is a derivation, and
\item $\delta_2$ satisfies $\delta_2(a b) = \delta_2(a) b + \delta_1(a) \delta_1(b) + a \delta_2(b)$.
\end{itemize}
The corresponding action is $A \to R \otimes A \colon a \mapsto \sum_{0 \leq i < 4} \varepsilon^i \otimes \delta_i(a)$,
where $\delta_0 := \id$ and $\delta_3 := \delta_1 \delta_2 = \delta_2 \delta_1$.
The quotient of $\Spec A$ by $\alpha_4$ (resp.\ its subgroup scheme $\alpha_2$)
is the spectrum of the subalgebra $\Ker \delta_1 \cap \Ker \delta_2$ (resp.\ $\Ker \delta_1$).

Now let $\delta_1$ be the one given in the proof 
and define $\map{\delta_2}{C_2}{C_2}$ by
$\delta_2(t, v, x, y, z) = (0, 1, y^2, z, 0)$
and the Leibniz-like formula.
Note that $\delta_2$ extends $\map{\delta'_2}{C_1}{C_1}$ given in the proof.
It is straightforward to check that this $(\delta_1,\delta_2)$ satisfy the required conditions.
\end{rem}

\begin{rem} \label{rem:covering of E_8^1}
The equation of $C_1$ in the case of $E_8^1$ of the proof of Theorem \ref{thm:smooth covering} is, after a coordinate change, 
of the form $x^2 + z^3 + z y^5 + x y^i + (\text{higher})$,
with $i = 6$ if $c \neq 0$ and $i = 7$ if $c = 0$.
In characteristic $0$, 
$k[[x,y,z]] / (x^2 + z^3 + z y^5)$ is an exceptional unimodal singularity usually denoted by the symbol $E_{13}$,
where the index $13$ stands for the Milnor number ($\dim_k k[[x,y,z]]/(F_x, F_y, F_z)$ for $k[[x,y,z]] / (F)$) in characteristic $0$,
although in characteristic $2$ it is not equal to the Milnor number nor the Tjurina number ($\dim_k k[[x,y,z]]/(F, F_x, F_y, F_z)$).
It is straightforward to check that our $C_1$ has Tyurina number $20$ if $c \neq 0$ and $22$ if $c = 0$.
\end{rem}

\subsection*{Acknowledgments}
I thank Hiroyuki Ito, Kentaro Mitsui, and Hisanori Ohashi for helpful comments and discussions.
I thank the anonymous referee for valuable comments and corrections.

\end{document}